\documentclass[a4paper,reqno,10pt]{amsart}
\usepackage[english]{babel}
\usepackage{amssymb,amsmath,hyperref}
\usepackage{bbold}
\usepackage{amsrefs}
\usepackage{bussproofs, epigraph, comment}
\usepackage{stackrel}
\usepackage{xcolor}

\newtheorem{thm}{Theorem}
\newtheorem{algo}[thm]{Algorithm}

\newtheorem{cor}[thm]{Corollary}
\newtheorem{defi}[thm]{Definition}
\newtheorem{rem}[thm]{Remark}
\newtheorem{nota}[thm]{Notation}

\newtheorem{ax}[thm]{Axiom}

\newtheorem{princ}[thm]{Principle}
\newtheorem{ack}[thm]{Acknowledgement}

\newtheorem*{question*}{Question}

\newcommand\be{\begin{equation}}
\newcommand\ee{\end{equation}}

\def\bdefi{\begin{defi}\rm}
\def\edefi{\end{defi}}
\def\bnota{\begin{nota}\rm}
\def\enota{\end{nota}}
\def\brem{\begin{rem}\rm}
\def\erem{\end{rem}}

\def\RCA{\textup{\textsf{RCA}}}

\def\WKL{\textup{\textsf{WKL}}}

\def\bye{\end{document}}

\def\N{{\mathbb  N}}

\def\({\textup{(}}
\def\){\textup{)}}

\def\st{\textup{st}}
\def\asa{\leftrightarrow}

\def\di{\rightarrow}

\def\M{\mathcal{M}}

\def\ACA{\textup{\textsf{ACA}}}

\def\paai{\Pi_{1}^{0}\textsf{\textup{-TRANS}}}
\def\Paai{\Pi_{1}^{1}\textsf{\textup{-TRANS}}}
\def\Deltat{\Delta_{1}^{1}\textsf{\textup{-TRANS}}}
\def\Deltac{\Delta_{1}^{1}\textsf{\textup{-CA}}}

\def\PRA{\textup{\textsf{PRA}}}
\def\QFAC{\textup{\textsf{QF-AC}}}
\def\PB{\textsf{PB}}

\def\EFA{\textup{\textsf{EFA}}}

\def\bdefi{\begin{defi}\rm}
\def\edefi{\end{defi}}
\def\bnota{\begin{nota}\rm}
\def\enota{\end{nota}}
\def\brem{\begin{rem}\rm}
\def\erem{\end{rem}}

\def\FIVE{\Pi_{1}^{1}\textup{-\textsf{CA}}_{0}}

\def\rel{\sqsubseteq}
\def\srel{\sqsubset}
\def\sler{\sqsupset}

\def\ATR{\textup{ATR}}

\def\INT{\textup{int}}

\def\RCAo{\textup{\textsf{RCA}}_{0}^{\omega}}
\def\RCAO{\RCA_{0}^{\Omega}}

\usepackage[foot]{amsaddr}

\begin{document}
\title{Searching through the reals}
\author{Sam Sanders}
\begin{abstract}
It is a commonplace to say that \emph{one can search through the natural numbers}, by which is meant the following:  For a property, decidable in finite time and which is not false for all natural numbers, checking said property 
starting at zero, then for one, for two, and so on, one will eventually find a natural number which satisfies the property, assuming no resource bounds.  By contrast, it seems one cannot search through the real numbers in any similarly `basic' fashion:  
The reals numbers are not countable, and their well-orders carry extreme logical strength compared to the basic notions involved in `searching through the natural numbers'.  
In this paper, we study two principles \textsf{(PB)} and \textsf{(TB)} from Nonstandard Analysis which essentially state that \emph{one can search through the reals}.  
These principle are \emph{basic} in that they involve only constructive objects of type zero and one, and the associated `search through the reals' 
amounts to nothing more than a bounded search \emph{involving nonstandard numbers as upper bound}, but independent of the \emph{choice} of this number.
We show that \textsf{(PB)} and \textsf{(TB)} are equivalent to known systems from the foundational program \emph{Reverse Mathematics}, namely respectively the existence of the hyperjump and $\Delta_{1}^{1}$-comprehension.  
We also show that \textsf{(PB)} and \textsf{(TB)} exhibit remarkable similarity to, respectively, 
 the Turing jump and recursive comprehension.    
In particular, we show that Nonstandard Analysis allows us to treat \emph{number quantifiers as `one-dimensional' bounded searches}, and \emph{set quantifiers as `two-dimensional' bounded searches}.  
\end{abstract}
\address{Department of Mathematics, Ghent University, Belgium \& Munich Center for Mathematical Philosophy, LMU Munich, Germany}
\email{sasander@me.com}
\maketitle
\vspace{-0.4cm}

\thispagestyle{empty}

\section{Introduction}
\subsection{Searching through the naturals and the reals}
It is a commonplace to say that \emph{one can search through the natural numbers}, by which is meant the following:  
\begin{quote}
For a property $Q(n)$, decidable in finite time and which is not false for all natural numbers, one successively checks if $Q(0), Q(1), Q(2), \dots$ holds, and one will eventually find a natural number $n$ such that $Q(n)$, assuming no further resource bounds.  
\end{quote}
In fact, Kleene defines the class of partial recursive functions as those obtained via primitive recursion plus the axiom \emph{Unbounded search}, and the latter exactly formalises the aforementioned informal description of `searching through the natural numbers';  We refer to  \cite{zweer}*{Def.\ 2.2, p.\ 10} for more details.  Furthermore, the semi-constructive \emph{Markov's principle} has a similar interpretation (See \cite{troelstra1}*{1.11.5}).    

\medskip

In contrast to the case of the natural numbers, it seems one cannot search through the real numbers in any remotely `basic' fashion:  
The reals numbers are not countable, and the existence of a well-order requires the axiom of choice.  Even fragments of the latter carry tremendous logical strength compared to the basic notions involved in `searching through the natural numbers';  See \cite{simpson2}*{Table 4} for a detailed overview of the strength of small fragments of the axiom of choice.  

\medskip

In this paper, we show that the framework of \emph{stratified Nonstandard Analysis} (See \cite{aveirohrbacek}) allows one to `search through the reals' in a rather basic fashion.  
In particular, we formulate a \emph{nonstandard} principle (\textsf{PB}) which essentially states that \emph{one can search through the reals}.  
The principle ($\PB$) is \emph{basic} in that it involves only \emph{constructive}\footnote{The exact meaning of `constructive' will be clarified in Section \ref{main1}.  The interpretation we have in mind is `acceptable in Bishop's \emph{Constructive Analysis}'.  See \cite{bish1} for the latter.\label{footkie}} objects of type $0$ and $1$, and the `search through the reals' amounts to nothing more than a bounded search through the natural numbers \emph{involving a nonstandard number as an upper bound}. 
It should be noted that the bounded search is \emph{independent of the choice of the nonstandard number}.  We show that \textsf{(PB)} is equivalent to a known principle, namely the Suslin functional (See e.g.\ \cite{kohlenbach2, avi2}).
Similarly, we formulate an analogous principle \textsf{(TB)} and prove equivalence to  $\Delta_{1}^{1}$-comprehension in functional form (See e.g.\ \cite{simpson2}*{I.11.8} for the latter).        

\medskip

As to the structure of this paper, we provide some more detailed motivation in Section \ref{motig}.  We introduce a suitable weak `base theory' in Section \ref{base} and recall known results. 
In Section \ref{main1}, we formulate the principle \textsf{(PB)} and prove its equivalence to the Suslin functional $(S^{2})$ over our base theory.  In Section \ref{main2}, we obtain similar results for $\Delta^{1}_{1}$-comprehension and the principle \textsf{(TB)}.    

\medskip

As to background information, $(S^{2})$ is the functional version of the strongest `Big Five' system $\FIVE$ studied in the foundational program \emph{Reverse Mathematics}.  The principle \emph{$\Delta_{1}^{1}$-comprehension} is also studied in the latter program.  
 We refer to \cite{simpson2,simpson1, avi2,kohlenbach2} for more details.  We do point out the following quote by Simpson:
\begin{quote}
From the above it is clear that the five basic systems $\RCA_{0}$, $\WKL_{0}$, $\ACA_{0}$, $\ATR_{0}$, $\FIVE$ arise naturally from investigations of the Main Question. 
The proof that these systems are mathematically natural is provided by Reverse Mathematics. (\cite{simpson2}*{p.\ 43}).     
\end{quote}
Hence, the principle \textsf{(PB)} is mathematically natural due to its equivalence to $(S^{2})$, the functional version of $\FIVE$.
Finally, we urge the reader to first consult Remark~\ref{ohdennenboom} so as to clear up a common misconception regarding Nelson's approach to Nonstandard Analysis.

\subsection{Motivation}\label{motig}
In this section, we discuss the background of, and more detailed motivation for, the topic of this paper.   
We first study the notion of `searching through the naturals' in Nonstandard Analysis.  
We only require very basic familiarity with Nelson's internal set theory, also introduced in Section \ref{base}.

\medskip

Firstly, we show that `searching through the naturals' amounts to a \emph{bounded} search in Nonstandard Analysis.  
To this end, recall that by Post's theorem a computable set can be described by a $\Delta_{1}^{0}$-formula, and vice versa (\cite{zweer}*{Theorem~2.2, p.\ 64}).  
Thus, consider the $\Delta_{1}^{0}$-formula, relative to `st', given by:
\be\label{valid}
(\forall^{\st}n^{0})[(\exists^{\st}k)f(n,k)=0\asa (\forall^{\st} m)g(n,m)\ne0].
\ee
Now define $p(n,h, M)$ as $(\mu k\leq M)h(n,k)=0$, if such exists and $M$ otherwise.  Then it is easy to show that for any infinite number $M^{0}$:
\be\label{karf}
(\forall^{\st}n^{0})\big[ (\exists^{\st}k)f(n,k)=0\asa p(n,f,M)\leq_{0}p(n,g, M)\big].
\ee
  Hence, to decide if a $\Delta_{1}^{0}$-formula (relative to `\st') holds, one need only \emph{perform a bounded search}, where the upper bound is any nonstandard number.


\medskip

Secondly, we show that \emph{relative to the Turing jump} `searching through the naturals' \emph{also} amounts to an explicit {bounded} search in Nonstandard Analysis, in contrast to the Turing jump's `oracle status'.    
To this end, consider the Turing jump functional:
\be\tag{$\exists^{2}$}
(\exists \varphi^{2})(\forall f^{1})\big[(\exists x^{0})f(x)=0 \asa \varphi(f)=0  \big], 
\ee
which by \cite{bennosam}*{Cor.\ 12} is equivalent over a version of \EFA~to
\be\tag{$\paai$}
(\forall^{\st} f^{1})\big[(\forall^{\st} x^{0})f(x)\ne0\di (\forall x)f(x)\ne0\big].
\ee
The latter is the Transfer principle from Nonstandard Analysis limited to $\Pi_{1}^{0}$-formulas.
As it turns out, $\paai$ provides a straightforward way to turn $\Pi_{1}^{0}$-formulas into \emph{bounded} formulas:  For standard $f^{1}$ possibly involving standard parameters and infinite $M^{0}$, $\paai$ implies that
\be\label{tranke}
(\forall x^{0})(f(x)\ne0)\asa (\forall x^{0}\leq_{0} M)(f(x)\ne0) 
\ee
Hence, to find a (standard) zero for standard $f^{1}$, one need only perform the bounded search $(\mu k\leq M)f(k)=0$.  
Furthermore, the latter (resp.\ the right-hand side of \eqref{tranke}) is elementary computable (resp.\ decidable) in terms of $f$ and $M$, and involves only objects of type zero besides $f$.     
This explicit nature, and the similarity to a $\Pi_{1}^{0}$-formula, should be contrasted to the right-hand side of $(\exists^{2})$.  
In other words, the right-hand side of \eqref{tranke} is much less of a `black box' than that of the `oracle' $(\exists^{2})$.    


\medskip

In short, the two previous examples suggest that `searching through the naturals' amounts to nothing more than a bounded search (involving an arbitrary nonstandard number) in Nonstandard Analysis.    
This search is `basic' in that it is given by an explicit formula, and is closely connected to the original formula.  

\medskip

The aim of this paper is to show that a \emph{similarly basic} `bounded search' in Nonstandard Analysis allows us to `search through the real numbers' using the algorithm $\mathfrak{(A)}$ defined in Section \ref{main1}.  We follow Kohlenbach (\cite{kohlenbach2}*{p.\ 289}) in assuming that any sequence of type one can be viewed as a real using his `hat function'.  Intuitively speaking, we shall establish that Nonstandard Analysis allows us to treat \emph{number quantifiers as `one-dimensional' bounded searches} (as in \eqref{karf} and \eqref{tranke}), 
and \emph{set quantifiers as `two-dimensional' bounded searches} (as in \eqref{bound} and \eqref{bound6}).  
In particular, we will formulate \textsf{(PB)} which constitutes a similar `bounding result' as in \eqref{tranke} \emph{generalised to $\Pi_{1}^{1}$-formulas}.
In the same way as $\paai$ is essential to \eqref{tranke}, the principle $\Paai$ is essential in establishing \textsf{(PB)}:
\begin{align}\label{paaikestoemp2}\tag{$\Paai$}
(\forall^{\st}f^{1})\big[(\forall^{\st} g^{1})(\exists^{\st}x^{0})&f(\overline{g}x)\ne0 \asa (\forall g^{1})(\exists x^{0})f(\overline{g}x)\ne0\big]
\end{align}
What is more, by \cite{bennosam}*{Cor.\ 15} and Theorem \ref{labbekak}, $\Paai$ is equivalent to \textsf{(PB)}, and to the \emph{Suslin functional}, defined as follows:
\be\label{suske}
(\exists S^{2})(\forall f^{1})\big[   S(f)=_{0} 0 \asa (\exists g^{1})(\forall x^{0}) (f(\overline{g}x)=0).     \big] \tag{$S^{2}$}
\ee
The Suslin functional is the `hyperjump' functional and corresponds to $\FIVE$, the strongest so-called Big Five system from \emph{Reverse Mathematics} (See \cite{simpson2}*{VI}).

\medskip

Inspired by the results regarding the Suslin functional, we obtain a similar bounding result \textsf{(TB)} for $\Delta_{1}^{1}$-comprehension using $\Deltat$, i.e.\ the Transfer principle limited to $\Delta_{1}^{1}$-formulas.  
In particular, we obtain a version of \eqref{karf} for $\Delta_{1}^{1}$-formulas to underline the analogy between standard sets and nonstandard numbers.  

\medskip

As to methodology, inspired by the bounding result \eqref{tranke}, we shall require that the bounded formula (equivalent to the $\Pi_{1}^{1}$ or $\Delta_{1}^{1}$-formula at hand) in \textsf{(PB)}, \textsf{(TB)}, and related principles, is \emph{basic}, by which we mean that it satisfies the following:
\begin{enumerate}
\renewcommand{\theenumi}{\Roman{enumi}}
\item Only type 0 and constructive$^{\ref{footkie}}$ type $1$ objects occur in the bounded formula.\label{conda}
\item The syntactic structure of the bounded formula is similar to that of the original $\Pi_{1}^{1}$ or $\Delta_{1}^{1}$-formula.\label{condb}
\end{enumerate}
With regard to condition \eqref{condb}, $\Pi_{1}^{1}$-formulas can be brought into the \emph{Kleene normal form} (See \cite{simpson2}*{V.1.4}).  The latter can be gleaned from the Suslin functional and we will 
directly work with this normal form.  Furthermore, it is clear that the well-known practice of `coding sets of numbers as nonstandard numbers' (See e.g.\ \cite{keisler1}) is not basic in our sense, as we deal with equivalent \emph{bounded} formulas.     

\medskip

Finally, to obtain the aforementioned results, we need to adopt the richer framework of \emph{stratified} Nonstandard Analysis developed by Hrbacek (\cites{hrbacek3, hrbacek4, hrbacek5, aveirohrbacek}) and pioneered by P\'eraire (\cite{peraire}).  
We briefly introduce this framework in Section \ref{base}

\section{Nonstandard Analysis}\label{base}
In this section, we define the system in which we shall prove the equivalences mentioned in the previous section  
We first introduce Nelson's \emph{internal set theory} and a suitable subsystem $\RCAO$ thereof in Section \ref{firstbase}.  We then introduce \emph{stratified Nonstandard Analysis} and $\RCA_{0}^{\dagger}$, a suitable extension of $\RCAO$, in Section \ref{secondbase}.    
%
\subsection{Nelson's syntactic Nonstandard Analysis}\label{firstbase}
In Nelson's \emph{internal set theory} (\cite{wownelly}), a \emph{syntactic} approach to Nonstandard Analysis as opposed to Robinson's semantic one (\cite{robinson1}), a new predicate `st($x$)', read as `$x$ is standard' is added to the language of \textsf{ZFC}.  
The notations $(\forall^{\st}x)$ and $(\exists^{\st}y)$ are short for $(\forall x)(\st(x)\di \dots)$ and $(\exists y)(\st(y)\wedge \dots)$.
The three axioms \emph{Idealization}, \emph{Standard Part}, and \emph{Transfer} govern the new predicate `st'  and give rise to a conservative extension of \textsf{ZFC}.   
Nelson's approach has been studied in the context of higher-type arithmetic in e.g.\ \cite{brie, bennosam, avi3}, and we single out one particular system, called $\RCAO$.  

\medskip

In two words, the system $\RCAO$ is a conservative extension of Kohlenbach's base theory $\RCAo$ from \cite{kohlenbach2} with certain axioms from Nelson's {Internal Set Theory} based on the approach from \cites{brie,bennosam}.    
This conservation result is proved in \cite{bennosam}, while certain partial results are implicit in \cite{brie}.  In turn, the system $\RCAo$ is a conservative extension of the base theory of Reverse Mathematics $\RCA_{0}$ for the second-order language by \cite{kohlenbach2}*{Prop.\ 3.1}.  
Following Nelson's approach in arithmetic, we define $\RCAO$ as the system 
\[
\textsf{E-PRA}_{\st}^{\omega*}+\textsf{QF-AC}^{1,0} +\textsf{HAC}_{\textsf{int}}+\textsf{I}+ \textsf{PF-TP}_{\forall} 
\]
from \cite{bennosam}*{\S3.2-3.3}.  To guarantee that the latter is a conservative extension of $\RCAo$, Nelson's axiom \emph{Standard part} must be limited to \textsf{HAC}$_{\INT}$, 
while Nelson's axiom \emph{Transfer} has to be limited to universal formulas \emph{without} parameters, as in \textsf{PF-TP}$_{\forall}$.  
On a technical note, the language $\RCAO$ actually involves a predicate $\st_{\rho}$ for every finite type $\rho$, but the subscript is always omitted.  
\begin{thm}\label{kerre}
The system $\RCAO$ is a conservative extension of $\RCA_{0}^{\omega}$.  
The system $\RCAO$ is a $\Pi_{2}^{0}$-conservative extension of ${\PRA}$.  
\end{thm}
\begin{proof}
See \cite{bennosam}*{Cor.\ 9}.  
\end{proof}
The conservation result for $\textsf{E-PRA}_{\st}^{\omega*}+\textsf{QF-AC}^{1,0}$ is trivial.  
Furthermore, omitting \textsf{PF-TP}$_{\forall}$, the theorem is implicit in \cite{brie}*{Cor.\ 7.6} as the proof of the latter goes through as long as \textsf{EFA} is available.

\medskip

The following theorem of $\RCAO$ is important.     
Note that the abbreviation `$M\in \Omega$' for $\neg\st(M^{0})$ is used.  The statement that \eqref{corkikiki} $\di$ \eqref{corkikiki2} for all such standard functionals is abbreviated $\Omega$-CA.  
If for a standard functional $F$, the functional $F(\cdot, M)$ satisfies \eqref{corkikiki}, we say the latter is \emph{$\Omega$-invariant}.  
\begin{thm}\label{genall}
In $\RCAO$, we have for all standard $F^{(\sigma\times 0)\di 0}$ that
\begin{align}
(\forall^{\st} x^{\sigma})(\forall M,N \in & \Omega)\big[F(x ,M)=F(x,N) \big]\label{corkikiki} \\
&\di (\exists^{\st}G^{\sigma\di 0})(\forall^{\st} x^{\sigma})(\forall N^{0}\in \Omega)\big[G(x)=_{0}F(x,N) \big].\label{corkikiki2}
\end{align}
\end{thm}
\begin{proof}
See \cite{firstHORM}*{\S2}.  
\end{proof}
\noindent
The `base theory' $\RCAO$ is quite useful in establishing equivalences, as is clear from the following theorem, which also establishes that the omission of parameters in \textsf{PF-TP}$_{\forall}$ is necessary (for obtaining a conservative extension as in Theorem \ref{kerre}).
\begin{thm}\label{markje}
The system $\RCAO$ proves $\paai\asa (\exists^{2})$.  Adding $\textup{\textsf{QF-AC}}^{1,1}$, we obtain $(S^{2})\asa \Paai$.  
\end{thm}
\begin{proof}
By \cite{bennosam}*{Cor.\ 12 and 15}.  
\end{proof}
Finally, the following theorem establishes that restricting the Standard Part principle as in \textsf{HAC}$_{\textup{int}}$ is necessary (for obtaining a conservative extension as in Theorem \ref{kerre}).
Let \textsf{WKL} be \emph{weak K\"onig's lemma} as in \cite{simpson2}*{IV} and let \eqref{STP} be 
\be\label{STP}\tag{\textup{\textsf{STP}}}
(\forall X^{1})(\exists^{\st} Y^{1})(\forall^{\st} x^{0})(x\in X\asa x\in Y).
\ee
Note that $\Omega$-CA for $\sigma=1$ is a version of the Standard Part Principle \eqref{STP}.
\begin{thm}
In $\RCAO+\eqref{STP}$, we have $\WKL$.
\end{thm}
\begin{proof}
See \cite{firstHORM}*{\S5}.  By way of a sketch, a standard binary tree with sequences of any standard length also contains a sequence of nonstandard length (by overspil or induction).  Apply \eqref{STP} to the latter sequence to obtain a standard path through the tree, and hence $\WKL^{\st}$.  
Rewrite $\WKL^{\st}$ as its contraposition, sometimes called \emph{fan theorem}, and apply \textsf{QF-AC}$^{1,0}$ relative to `st' (which follows from \textsf{HAC}$_{\textup{int}}$) to the antecedent.  Drop all `\st' in the antecedent and consequent of the innermost implication, and apply \textsf{PF-TP}$_{\forall}$ to yield \textsf{WKL}.    
\end{proof}
By \cites{keisler1, briebenno}, \eqref{STP} actually yields a conservative extension of $\WKL_{0}$ from \cite{simpson2}*{IV}.  

\subsection{Stratified Nonstandard Analysis}\label{secondbase}
The framework of  \emph{Stratified Nonstandard Analysis} (\cite{hrbacek3, hrbacek4, hrbacek5, aveirohrbacek, peraire}) is a refinement of Nelson's where the unary standardness predicate `$\st(x)$' is replaced by the binary predicate `$x\rel y$', read as `$x$ is standard relative to $y$' and $x\rel 0$ is still read `$x$ is standard' or `st$(x)$'.    
We denote $\neg(y\rel x)$ by $x\srel y$ and say that `$y$ is nonstandard relative to $x$'.  

\medskip

In the same way, extend the language of $\RCAO$ with new predicates $\rel_{\rho, \tau}$, one for each pair of finite types.  We will often omit the subscript as is common for the standardness predicate of $\RCAO$.  
The axioms of $\RCAO$ govern the predicate $\st(x)$, whereas the following basic axioms govern $x\rel y$.  
\begin{ax}[\textsf{BASIC}]~ 
\begin{enumerate}
\renewcommand{\theenumi}{\roman{enumi}}
\item $(\forall x)[\st(x)\asa x\rel 0]$.
\item $(\forall x)(0\rel x \wedge x\rel x)$.  
\item $(\forall x, y,z )\big[ (x\rel y\wedge y\rel z)   \di x\rel z \big]$.\label{three}
\item $(\forall x)(\exists y^{0},z^{0})(x\srel y \wedge x\rel z)$.
\item $(\forall x^{\sigma\di \tau}, y^{\sigma},z)(x,y\rel z\di x(y)\rel z)$.\label{final}
\end{enumerate}
\end{ax}    
\noindent
The \textsf{BASIC} axioms are rather elementary and express the following facts:
\begin{enumerate}
\renewcommand{\theenumi}{\roman{enumi}}
\item Being standard is the same as being standard relative to zero.
\item All objects are standard relative to themselves.  Zero is the `least' level of standardness.
\item Transitivity holds for `being standard relative to'.
\item Every level of standardness is inhabited by a number.  There always exists a similarly inhabited higher level.  
\item Functional application preserves relative standardness.  
\end{enumerate}

Denote $\RCA_{0}^{\dagger}$ as $\RCAO+\textsf{BASIC}$ in the language extended by `$\rel$'.
Clearly $\RCA_{0}^{\dagger}$ is only a definitional extension of $\RCAO$, i.e.\ the former is also a conservative extension of $\RCA_{0}$ and $\PRA$ similar to Theorem \ref{kerre}.    
We also require the following {Standard part principle}, not stronger than \eqref{STP}.  
\be\label{S}\tag{{\textup{\textsf{STP2}}}}
(\forall X^{1},z)(\exists Y^{1}\rel z)(\forall x^{0}\rel z)(x\in Y\asa x\in X).
\ee
We finish this section with an important remark about the internal framework.  
\begin{rem}\label{ohdennenboom}\rm
Tennenbaum's theorem (\cite{kaye}*{\S11.3}) `literally' states that any nonstandard model of \textsf{PA} is not computable.  \emph{What is meant} is that for a nonstandard model $\M$ of \textsf{PA}, the operations $+_{\M}$ and $\times_{\M}$ cannot be computably defined in terms of the operations $+_{\N}$ and $\times_{\N}$ of the standard model $\N$ of \textsf{PA}.  

\medskip

While Tennenbaum's theorem is of interest to the \emph{semantic} approach to Nonstandard Analysis involving nonstandard models, $\RCAO$ is based on Nelson's \emph{syntactic} framework, and therefore Tennenbaum's theorem does not apply:  Any attempt at defining the (external) function `$+$ limited to the standard numbers' is an instance of \emph{illegal set formation}, forbidden in Nelson's \emph{internal} framework (\cite{wownelly}*{p.\ 1165}).  

\medskip

To be absolutely clear, lest we be misunderstood, Nelson's \emph{internal set theory} \textsf{IST} forbids the formation of \emph{external} sets $\{x\in A: \st(x)\}$ and functions `$f(x)$ limited to standard $x$'.  
Therefore, any appeal to Tennenbaum's theorem to claim the `non-computable' nature of $+$ and $\times$ from $\RCAO$ is blocked, for the simple reason that the functions `$+$ and $\times$ limited to the standard numbers' \emph{simply do not exist}.              
On a related note, we recall Nelson's dictum from \cite[p.\ 1166]{wownelly} as follows:
\begin{quote}
\emph{Every specific object of conventional mathematics is a standard set.} It remains unchanged in the new theory \textup{[\textsf{IST}]}.  
\end{quote}
In other words, the operations `$+$' and `$\times$', but equally so primitive recursion, in (subsystems of) \textsf{IST}, are \emph{exactly the same} familiar operations we  
know from (subsystems of) \textsf{ZFC}.  Since the latter is a first-order system, we however cannot exclude the presence of nonstandard objects, and internal set theory just makes this explicit, i.e.\ \textsf{IST} turns a supposed bug into a feature.    
\end{rem}

\section{Main results}\label{main}
\subsection{Stratified bounding and the Suslin functional}\label{main1}
In this section, we formulate the bounding principle \textsf{(PB)} and prove its equivalence to the Suslin functional.
We also prove that \textsf{(PB)} gives rise to the algorithm $\mathfrak{(A)}$ for finding witnesses to $\Sigma_{1}^{1}$-formulas relative to the standard world.  

\medskip

First of all, we prove some `relative' versions of the Transfer principle.
\begin{thm}\label{lreall}
In $\RCA_{0}^{\dagger}$, $\paai$ is equivalent to 
\[
(\forall z)(\forall f^{1}\rel z)\big[(\forall x^{0}\rel z)f(x)\ne0\di (\forall x)f(x)\ne0\big],
\]
and $\Paai$ is equivalent to 
\[
(\forall z)(\forall f^{1}\rel z)\big[(\forall  g^{1}\rel z)(\exists x^{0})f(\overline{g}x)\ne0 \leftrightarrow (\forall g^{1})(\exists x^{0})f(\overline{g}x)\ne0\big]
\]
\end{thm}
\begin{proof}
Clearly, the reverse implications follow from \textsf{BASIC} and taking $z=0$.
By \cite{bennosam}*{Cor.\ 12}, $\paai$ is equivalent to the following sentence:
\be\tag{$\mu^{2}$}
(\exists \mu^{2})(\forall f^{1})\big[(\exists x^{0})f(x)=0\di f(\mu(f))=0\big].
\ee
Since $(\mu^{2})$ does not involve parameters, we may apply (the contraposition of) \textsf{PF-TP}$_{\forall}$ to $(\mu^{2})$, and hence assume that $\mu$ as in $(\mu^{2})$ is standard.  
Thus, we have $\mu \rel z $ for any $z$ by axiom \eqref{three} in \textsf{BASIC}.  By axiom \eqref{final} in the latter, for any $f\rel z$, we have $\mu(f)\rel z$.  
By the definition of $(\mu^{2})$, if $(\exists x^{0})f(x)=0$, then $(\exists x^{0}\rel z)f(x)=0$, which is what we needed to prove for $\paai$.  One proceeds in exactly the same way for $\Paai$, as the latter is equivalent to 
\be\label{forex}
(\exists \nu^{1\di 1})(\forall f^{1})\big[(\exists g^{1})(\forall x^{0})(f(\overline{g}x)=0)\di (\forall x^{0})(f(\overline{\nu(f)}x)=0) \big],
\ee 
by \cite{bennosam}*{Cor.\ 15}, and we are done.
\end{proof}\noindent
The functional \eqref{forex} is called $(\mu_{1})$ in \cite{avi2}, and is equivalent to $(S^{2})$ assuming \textsf{QF-AC}$^{1,1}$.   
It is clear that $(\mu_{1})$ and $(S^{2})$ do not satisfy either of the conditions \eqref{conda} and \eqref{condb}.  

\medskip

Secondly, we consider an important consequence of the idealization axiom \textsf{I}.  
\begin{thm}\label{hake}
In $\RCAO$, there is a \(nonstandard\) function $h_{0}^{1}$ which dominates all standard $f^{1}$ everywhere, i.e.\ $(\forall^{\st}f^{1})(\forall n^{0})(f(n)\leq_{0}h_{0}(n))$ or  $(\forall^{\st}f^{1})(f\leq_{1} h_{0})$ .
\end{thm}
\begin{proof}
Note that the following formula is trivially true:
\be\label{frik}
(\forall^{\st} g^{1^{*}})(\exists h^{1})(\forall k^{1}\in g)\big[(\forall x^{0})(k(x)\leq_{0} h(x))\big], 
\ee
where `${1^{*}}$' is the type of sequences (with length of type $0$) of type 1 objects.  
The formula in square brackets in \eqref{frik} is internal and applying idealization \textsf{I} yields:
\[
(\exists h^{1})(\forall^{\st}g^{1})\big[(\forall x^{0})(g(x)\leq_{0} h(x))\big].
\]
The function $h$ is as required for the theorem.
\end{proof}
\begin{rem}[Constructive idealization]\rm
We shall refer to the function $h_{0}$ from Theorem \ref{hake} as `constructive' for the following reason:  The proof of Theorem \ref{hake} trivially goes though in the system \textsf{H} from \cite{brie}*{\S5.2}, which is a 
conservative extension of Heyting arithmetic with among other axioms \textsf{I} (See \cite{brie}*{Cor.\ 5.6}).  Hence, the existence of $h_{0}$ is \emph{constructively} acceptable, in that the axiom \textsf{I} included in the system \textsf{H} results in a conservative extension of Heyting arithmetic (in the original language).  
Heyting arithmetic in all finite types is only a small fragment of the usual\footnote{We have in mind such systems as \textsf{CZF} and Martin-L\"of Type Theory (\cites{ML, aczelrathjen}).} systems providing a foundation for Bishop's \emph{Constructive Analysis} (\cite{bish1}).  Hence, we may refer to the function $h_{0}$ as `constructive (in the sense of Bishop)'.    
%
%
\end{rem}
%
%
%
Thirdly, we formulate our long-awaited bounding principle \textsf{(PB)}.  The function $h_{0}$ therein is intended to be the one from the previous theorem.  
Recall also the definition of `$\tau^{0}\leq_{0^{*}}\sigma^{0}$' as $|\sigma|=|\tau|\wedge (\forall i< |\sigma|)(\tau(i)\leq_{0} \sigma(i))$.  

\begin{princ}[\textsf{PB}] There is $h_{0}\sler 0$ such that for $f^{1}\rel 0$, $h\geq_{1}h_{0}$ and $M\sler h$,
\be\label{bound}
(\forall^{\st}g^{1})(\exists^{\st} x^{0})(f(\overline{g}x)\ne0) \asa (\forall g^{0}\leq_{0^{*}}\overline{h}M)(\exists x^{0}\leq M)(f(\overline{g}x)\ne0).
\ee
\end{princ}
Fourth, we prove the following theorem.  By \cite{yamayamaharehare}*{Theorem 2.2} and \cite{keisler1}, the base theory is weak, i.e.\ certainly not stronger than $\ACA_{0}$ and $\WKL_{0}$ respectively.  
\begin{thm}\label{labbekak}
In $\RCA_{0}^{\dagger}+\eqref{S}$, we have $(\mu_{1}) \asa\textup{(\textsf{PB})}\asa \Paai$.  \\
In $\RCA_{0}^{\dagger}+\eqref{S}+\textup{QF-AC}^{1,1}$, we have $(S^{2}) \asa\textup{(\textsf{PB})}\asa \Paai$.
\end{thm}
\begin{proof}
We establish $\Paai\asa (\textup{PB})$ in $\RCA_{0}^{\dagger}+\eqref{S}$ using Theorem \ref{lreall}, and the theorem is then immediate by \cite{bennosam}*{Cor.\ 15}.

\medskip

In order to prove $\Paai \di$ (PB), consider $h_{0}$ from Theorem \ref{hake}, assume $(\forall^{\st}g^{1})(\exists^{\st} x^{0})(f(\overline{g}x)\ne0)$ for $f\rel 0$ and apply $\Paai$ to obtain $(\forall g^{1}\rel h_{0})(\exists x^{0}\rel h_{0})(f(\overline{g}x)\ne0)$.  
Now consider $g_{0}\leq_{1} h_{0}$ (which may or may not satisfy $g_{0}\rel h_{0}$) and apply \eqref{S} to obtain $g_{1}\rel h_{0}$ such that: 
\be\label{mondayspecial}
(\forall x^{0}\rel h_{0})(g_{0}(x)= g_{1}(x)).  
\ee
Since we already proved $(\forall g^{1}\rel h_{0})(\exists x^{0}\rel h_{0})(f(\overline{g}x)\ne0)$, we obtain $(\exists x_{0}\rel h_{0})(f(\overline{g_{1}}x_{0})\ne0)$.  
By \eqref{mondayspecial}, we also get $(\exists x_{0}\rel h_{0})(f(\overline{g_{0}}x_{0})\ne0)$, as $\overline{g_{0}}z=_{0}\overline{g_{1}}z\rel h_{0}$ for any $z\rel h_{0}$.     
Hence, we have proved that 
\[
(\forall g^{1}\leq_{1} h_{0} )(\exists x^{0}\rel h_{0})(f(\overline{g}x_{0})\ne0),
\]
and for $h_{0}\srel M^{0}$, we obtain:
\be\label{Sgrang}
(\forall g^{1}\leq_{1} h_{0} )(\exists x^{0}\leq M)(f(\overline{g}x_{0})\ne0).
\ee
By definition, \eqref{Sgrang} now yields: 
\be\label{Sram}
(\forall g^{0}\leq_{0^{*}} \overline{h_{0}}M )(\exists x^{0}\leq M)(f(\overline{g}x_{0})\ne0).
\ee
Indeed, for $g^{0}\leq_{0^{*}} \overline{h_{0}}M$, define $l^{1}:=g*00\dots$ and apply \eqref{Sgrang} in light of $l\leq_{1}h_{0}$.  Now repeat the above steps for any $h\geq_{1}h_{0}$ instead of $h_{0}$ to obtain the forward implication in \eqref{bound}.  

\medskip

Now assume the formula \eqref{Sram} for $M\sler h_{0}$ and $h_{0}$ as in the first paragraph of this proof, and consider standard $g^{1}$.  By the definition of $h_{0}$, we have $g\leq_{1}h_{0}$, implying $\overline{g}M\leq_{0^{*}} \overline{h_{0}}M$.
By assumption, we have $(\exists x^{0}\leq M)(f(\overline{\overline{g}M}x_{0})\ne0)$, which immediately yields $(\exists x^{0}\leq M)(f(\overline{g}x_{0})\ne0)$ and also $(\exists x^{0}\rel M)(f(\overline{g}x_{0})\ne0)$.  Applying $\paai$ yields  $(\exists x^{0}\rel 0)(f(\overline{g}x_{0})\ne0)$, and we have proved $(\forall^{\st}g^{1})(\exists^{\st} x^{0})(f(\overline{g}x)\ne0)$.  The equivalence \eqref{bound} now follows.  

\medskip

For the implication $\textup{\textsf{(PB)}}\di \Paai$, note that \textsf{(PB)} implies $\paai$, which immediately yields the reverse direction in $\Paai$.  
To prove the remaining implication in the latter, assume $(\forall^{\st}g^{1})(\exists^{\st} x^{0})(f(\overline{g}x)\ne0)$ for standard $f$, and let $h_{0}$ be the function from \textsf{(PB)}.  
Fix $g_{1}^{1}$ and define $h^{1}$ by $h(n):=\max(h_{0}(n), g_{1}(n))$.  Clearly, $h\geq_{1}h_{0}$, yielding $(\forall g^{0}\leq_{0^{*}}\overline{h}M)(\exists x^{0}\leq M)(f(\overline{g}x)\ne0)$ by \textsf{(PB)} for $M\sler h$.  
Hence, for $g^{0}_{0}=\overline{g_{1}}M$, we obtain $(\exists x^{0}\leq M)(f(\overline{g_{0}}x)\ne0)$, implying $(\exists x^{0})(f(\overline{g_{1}}x)\ne0)$.  Then $(\forall g^{1})(\exists  x^{0}\rel g)(f(\overline{g}x)\ne0)$ by $\paai$ and the 
forward implication in $\Paai$ also holds.  
\end{proof}
Comparing \eqref{tranke} and \eqref{bound}, we note that Nonstandard Analysis allows us to treat \emph{type zero quantifiers as `one-dimensional' bounded searches}, 
and \emph{type one quantifiers as `two-dimensional' bounded searches}.  

\medskip

It is then a natural question, originally due to Dag Normann, if \eqref{bound} allows one to \emph{find} a standard $g^{1}$ such that $(\forall^{\st} x^{0})(f(\overline{g}x)=0)$, assuming such exists?  
Now, the formula \eqref{bound} from \textsf{(PB)} suggests the following algorithm to solve this question. 
As above, we fix $h_{0}$ as in Theorem~\ref{hake} and $M\sler h_{0}$.  
\begin{algo}[$\mathfrak{A}$]
Check in lexicographical order starting with $\sigma=00\dots 00$ the formula $A(\sigma)\equiv(\forall x^{0}\leq M)(f(\overline{\sigma}x)=0)$ for all $\sigma$ of length $M$ and bounded above by $\overline{h}_{0}M$ .  
Output the lexicographically first $\sigma_{0}$ satisfying $A(\sigma_{0})$ if such there is, and $M_{0}\dots M_{0}$ otherwise, where $M_{0}=\max_{i\leq M}h_{0}(i)$.      
\end{algo}
Let us denote by $\mathfrak{A}(f)$ the output of the previous algorithm on input $f^{1}$. 
Note that there is no a priori reason why $\mathfrak{A}(f)$ even outputs a sequence with a standard part, i.e.\ such that $(\forall^{\st}n^{0})\st(\mathfrak{A}(f)(n))$, as the lexicographical order places lots\footnote{For instance, the sequence $0M00\dots 00$ comes before $n00\dots00$ for any infinite $M^{0}$ and standard $n^{0}$ in the lexicographical order.  In general, there does not seem to be an \emph{internal} ordering in which all the type zero sequences with standard part come first.} of sequences \emph{without a standard part} before those with one.  

\medskip

%
%
The previous observation notwithstanding, the following corollary shows that for standard $f^{1}$ and assuming \textsf{(PB)}, the algorithm $\mathfrak{(A)}$ always outputs a witness to $(\exists^{\st}g^{1})(\forall^{\st} x^{0})(f(\overline{g}x)=0)$ if and only if the latter formula holds.  
In other words, $\mathfrak{(A)}$ finds the required standard witness if such exists.  
\begin{cor}\label{FRS}
In $\RCA_{0}^{\dagger}+\eqref{S}+\textup{(\textsf{PB})}$, we have
\be\label{travi}
(\forall^{\st}f^{1})\big[(\forall^{\st}n^{0})\st(\mathfrak{A}(f)(n))\asa (\exists^{\st}g^{1})(\forall^{\st}x^{0})f(\overline{g}x)=0\big].  
\ee
\end{cor}
\begin{proof}
The forward direction in \eqref{travi} is immediate due to \eqref{STP}.  
For the reverse direction, by the theorem, we may use $\Paai$; Now suppose that for some standard $f_{0}$, we have the right-hand side of \eqref{travi} and the sequence $\mathfrak{A}(f_{0})$ is such that $(\exists^{\st}n_{0})\neg\st(\mathfrak{A}(f_{0})(n_{0}+1))$.  
For now, we assume that $n_{0}$ is \emph{the least such} number, and later prove that such a least number indeed exists using \eqref{STP}.  
By our assumption, the sequence $\overline{\mathfrak{A}(f_{0})}n_{0}$ is standard and we have the following formula:
\be\label{longue}
(\exists \sigma_{1}^{0}\leq_{0^{*}}\overline{h_{0}}M )\big[(\forall x^{0}\leq M)f_{0}(\overline{\sigma_{1}}x)=0\wedge \overline{\sigma_{1}}n_{0}\leq_{0^{*}} \overline{\mathfrak{A}(f_{0})}n_{0} \big].
\ee
Since $(\forall n^{0}\rel h_{0})(\sigma(n)\leq h_{0}(n))$, we can apply \eqref{S} for $z=h_{0}$ and obtain $g\rel h_{0}$ which is the standard part of $\sigma_{1}$ in the previous formula.  Hence, \eqref{longue} yields: 
\be\label{incheek}
(\exists g^{1}\rel h_{0})\big[(\forall x^{0}\rel h_{0})f_{0}(\overline{g}x)=0\wedge \overline{g}n_{0}\leq_{0^{*}} \overline{\mathfrak{A}(f_{0})}n_{0} \big].
\ee
By (a trivial variation of) Theorem \ref{lreall}, the previous formula and $\Paai$ yield:
\be\label{tongue}
(\exists^{\st} g_{1}^{1})\big[(\forall x^{0})f_{0}(\overline{g_{1}}x)=0\wedge \overline{g_{1}}n_{0}\leq_{0^{*}} \overline{\mathfrak{A}(f_{0})}n_{0} \big].
\ee
Note that we are allowed to apply Transfer as $\overline{\mathfrak{A}(f_{0})}n_{0}$ is a \emph{standard} parameter in \eqref{incheek} by the axioms $\mathcal{T}_{\st}^{*}$ of $\RCAO$ (See \cite{brie}*{Def.\ 2.2 and Lemma~2.8}).
Alternatively, use \eqref{STP} to obtain the standard part of $\overline{\mathfrak{A}(f_{0})}n_{0}*00\dots$ and replace the latter by the former in \eqref{incheek}.  

\medskip

However, for the standard $g^{1}_{1}$ as in \eqref{tongue}, we also have $g_{1}(n_{0}+1)<\mathfrak{A}(f_{0})(n_{0}+1)$, as the former number is finite, and the latter infinite.  
Hence, it is clear that $\overline{g_{1}}M$ comes before $\mathfrak{A}(f_{0})$ in the lexicographical ordering, and $(\mathfrak{A})$ should have output $\overline{g_{1}}M$ by \eqref{tongue}.  
This contradiction proves the theorem, modulo our assumption on $n_{0}$.  
To prove the latter assumption, we will prove the following version of \emph{external induction} using \eqref{STP}:  
\be\label{EI}
(\forall f^{1})\big[[ \st(f(0))\wedge (\forall^{\st}m)(\st(f(m))\di \st(f(m+1)))  ]\di  (\forall^{\st}n)\st(f(n)) \big].
\ee
To this end, fix $f^{1}$ satisfying the antecedent of \eqref{EI} and define the set $X^{1}$ as $\{(n,f(n)): n=n\} $ (using the well-known (standard) coding of pairs).  Using \eqref{STP}, let $Y^{1}$ be the standard part of $X$, and let $Z$ be the projection of $Y$ onto the first coordinate.  
Then $0\in Y$ and $(\forall^{\st}m^{0})[m\in Y \di m+1\in Y]$ by the antecedent of \eqref{EI}.  By (quantifier-free) induction, we have $(\forall^{\st}n^{0})(n\in Y)$, implying the consequent of \eqref{EI}.    
Finally, note that our assumption on $n_{0}$ from the previous paragraph of this proof, follows from the contraposition of \eqref{EI}.
\end{proof}
In light of the above, it is straightforward to formulate a version of \eqref{travi} \emph{equivalent} to \textsf{(PB)} using $h\geq_{1}h_{0}$ and $M\sler h$ as in the latter.  

\medskip

In conclusion, we have proved that \textsf{(PB)} gives rise to the algorithm $\mathfrak{(A)}$ for finding witnesses to $\Sigma_{1}^{1}$-formulas relative to the standard world.  
In particular, Corollary~\ref{FRS} establishes that \textsf(PB) expresses that \emph{one can search through the reals}.  
%
%
%
We end this section with a highly relavant note on generalisations of \textsf{(PB)}.
\begin{rem}[Generalisations]\label{genpb}\rm
Above, we proved \textsf{(PB)} for standard functions $f^{1}$, but the proof of Theorem~\ref{labbekak} easily generalises to any $f$ using the following `relativised idealisation':  
\be\tag{$\textup{\textsf{rI}}$}
(\forall z)\big[ (\forall x^{\sigma^{*}} \rel z)(\exists y^{\tau})(\forall x'\in x)\varphi(x', y)\di (\exists y^{\tau})(\forall x^{\sigma}\rel z)\varphi(x,y)   \big].
\ee
Indeed, in the same way as in the proof of Theorem \ref{hake}, the axiom \textsf{(rI)} easily yields a function $h^{1}\sler z$ which dominates all functions $f^{1}\rel z$ everywhere, i.e.\ $(\forall f^{1}\rel z)(\forall n^{0})(f(n)\leq_{0} h(n))$.  
It is then easy to obtain a version of \textsf{(PB)} for `$\rel 0$' replaced by `$\rel z$' following the proof of Theorem \ref{labbekak}.  This version would be as follows:
\begin{princ}[\textsf{rPB}] For all $z$, there is $h_{0}\sler z$ such that for $f^{1}\rel z$, $M\sler h\geq_{1}h_{0}$\textup{:}
\be\label{boundr}
(\forall g^{1}\rel z)(\exists x^{0}\rel z)(f(\overline{g}x)\ne0) \asa (\forall g^{0}\leq_{0^{*}}\overline{h}M)(\exists x^{0}\leq M)(f(\overline{g}x)\ne0).
\ee
\end{princ}
A generalised version of the algorithm $\mathfrak{(A)}$ can now be read off from \textsf{(rPB)}.  
Finally, in the same was as for \cite{brie}*{Cor.\ 7.8}, one proves that \textsf{(rI)} yields a conservative extension of Peano Arithmetic (and the same for fragments at least \EFA).  These generalisations are straightforward and we therefore do not go into details.    
\end{rem}
\subsection{Stratified bounding and $\Delta_{1}^{1}\textup{-\textsf{CA}}_{0}$}\label{main2}
In this section, we establish a bounding result like \textsf{(PB)} for the system $\Delta_{1}^{1}$\textsf{-CA}$_{0}$ (See \cite{simpson2}*{I.11.8}).  
In light of the similarities between the nonstandard treatment of $\Pi_{1}^{0}$ and $\Delta_{1}^{0}$-formulas in Section \ref{motig}, such a result is expected.  
We use the abbreviation $D(f, g)$ for the following formula, expressing that $f, g$ give rise to a $\Delta_{1}^{1}$-formula:
\[
(\exists h^{1})(\forall x^{0})[f(\overline{h}x)=0]\asa (\forall k^{1})(\exists y^{0})[g(\overline{k}y)\ne0]. 
\]
First of all, consider the following comprehension and transfer principle:
\be\tag{$\Deltat$}
(\forall f^{1}, g^{1})\big[ D(f, g)\di [D(f,g)\asa (\exists^{\st} h^{1})(\forall^{\st} x^{0})(f(\overline{h}x)=0)]\big],
\ee
\be\tag{$D_{2}$}
(\exists \Phi^{(1\times1)\di 1})(\forall f^{1}, g^{1})\big[ D(f, g)\di [\Phi(f, g)=0\asa (\exists h^{1})(\forall x^{0})(f(\overline{h}x)=0)]\big].
\ee
Clearly, $(D_{2})$ is the functional version of $\Deltac$, and we have the following theorem.    
%
\begin{thm}
In $\RCAO+\eqref{STP}+\QFAC^{1,1}$, we have $\Deltat\asa (D_{2})$.
\end{thm}
\begin{proof}
Clearly, both principles imply $\paai$ and $(\exists^{2})$.  Assume $(D_{2})$ and drop the reverse implication in the consequent.  
In the resulting formula, bring all type one-quantifiers to the front, which results in a formula of the form $(\exists \Phi)\psi(\Phi)$ where $\psi\in \Pi_{2}^{1}$.  The existential set-quantifiers in $\psi(\Phi)$ originate from the $(\exists h^{1})$ in the consequent of $(D_{2})$, and the reverse implication in $D(f,g)$.  The universal quantifiers in $\psi(\Phi)$ originate from $(\forall f^{1}, g^{1})$ and from the forward implication in $D(f,g)$.     
Now use $(\exists^{2})$ to remove arithmetical quantifiers in $\psi(\Phi)$ and apply $\QFAC^{1,1}$ to obtain $\Xi$ witnessing the existential set-quantifiers in $\psi(\Phi)$.  One of the components of $\Xi$, say the first one, witnesses the existential quantifier which originated from the $(\exists h^{1})$ quantifier in the consequent of $(D_{2})$;  We ignore the other components of $\Xi$, and obtain the following, thanks to the definition of $D(f, g)$:   
\be\label{hrux}
(\exists \Phi, \Xi)(\forall f^{1}, g^{1})\big[ D(f, g)\di [\Phi(f, g)=0\di  (\forall x^{0})(f(\overline{\Xi(1)(f, g, f, g)}x)=0)]\big].  
\ee
Since the previous formula is parameter-free, we may assume $\Phi$ and $\Xi$ are standard by \textsf{PF-TP}$_{\forall}$.  
Hence, if $D(f, g)$ for standard $f^{1}, g^{1}$, $\Xi(f, g, f, g)$ is a standard witness for the left-hand side of $D(f,g)$, if this side holds, and $\Deltat$ follows.    

\medskip

Now assume $\Deltat$ and note that the latter and $\paai$ imply:
\[
(\forall^{\st} f^{1}, g^{1})(\exists^{\st}l^{1})\big[ D(f, g)\di \big[(\exists h^{1})(\forall x^{0})(f(\overline{h}x)=0)\di  (\exists h^{1}\leq_{1} l^{1})(\forall x^{0})(f(\overline{h}x)=0)\big]\big].
\]
Apply \textsf{HAC}$_{\textup{int}}$ to obtain standard $\Psi^{(1\times 1)\di 1^{*}}$ such that $(\exists l\in \Psi(f,g))$.  Define $\Phi^{(1\times1)\di 1}$ as follows: $\Phi(f,g)(n):=\max_{i<|\Psi(f,g)|}\Psi(f,g)(i)(n)$.  
Clearly, we have for all standard $f^{1},g^{1}$ that if $D(f, g)$ then
\be\label{kurfd}
(\exists h^{1})(\forall x^{0})[f(\overline{h}x)=0]\di  (\exists h^{1}\leq_{1} \Phi(f,g))(\forall x^{0})(f(\overline{h}x)=0)], 
\ee
and the reverse implication is trivial.  By \eqref{STP} and $(\exists^{2})$, the consequent of \eqref{kurfd} is equivalent to $(\exists h^{0}\leq_{0^{*}} \overline{\Phi(f,g)}M)(\forall x^{0}\leq M)(f(\overline{h}x)=0)]$ for any infinite $M$.  
Hence, with the same assumptions in place, we obtain:
\[
(\exists^{\st} h^{1})(\forall^{\st} x^{0})[f(\overline{h}x)=0]\asa (\exists h^{0}\leq_{0^{*}} \overline{\Phi(f,g)}M)(\forall x^{0}\leq M)(f(\overline{h}x)=0)].
\]
Using $\Omega$-CA, we obtain the functional as in $(D_{2})^{\st}$.  The latter implies $(D_{2})$ in the same way that $(S^{2})\asa (S^{2})^{\st}$ in the proof of \cite{bennosam}*{Cor.\ 15}.   
\end{proof}
We now prove a result similar to Theorem \ref{lreall} for $\Deltat$.  
\begin{cor}
In $\RCA_{0}^{\dagger}$, $\Deltat$ is equivalent to its relativised version:
\[
(\forall z)(\forall f^{1}, g^{1}\rel z)\big[ D(f, g)\di [D(f,g)\asa (\exists h^{1}\rel z)(\forall x^{0}\rel z)(f(\overline{h}x)=0)]\big].
\]
\end{cor}
\begin{proof}
Immediate from \eqref{hrux}.
\end{proof}
Now define the following versions of \textsf{(PB)} as follows:
\begin{princ}[\textsf{SB}] There is $h_{0}\sler 0$ such that for $f^{1}, g^{1}\rel 0$ with $D(f,g)$, $l^{1}\geq_{1}h_{0}$, and $M^{0}\sler l$, we have
\be\label{bound5}
(\forall^{\st}k^{1})(\exists^{\st} x^{0})(f(\overline{k}x)\ne0) \asa (\forall k^{0}\leq_{0^{*}}\overline{l}M)(\exists x^{0}\leq M)(f(\overline{k}x)\ne0).
\ee
\end{princ}
Let $P(k^{1}, h^{1},M^{0})$ be the lexicographically least sequence $\sigma^{0}\leq_{0^{*}}\overline{h}M$ of length $M$ such 
that $(\forall x^{0}\leq M)(k(\overline{\sigma}x)=0)$, if such exists and $M_{0}...M_{0}$ (of length $M$) otherwise, where $M_{0}$ is the maximum of $h(i)+1$ for $i\leq M$.  
The following principle should be compared to \eqref{karf} in Section~\ref{motig}.
\begin{princ}[\textsf{TB}] There is $h_{0}\sler 0$ such that for all $f^{1}, g^{1}\rel 0$ with $D(f,g)$, all $l^{1}\geq_{1}h_{0}$, and all $M^{0}\sler l$, we have
\be\label{bound6}
(\exists^{\st}k^{1})(\forall^{\st} x^{0})(f(\overline{k}x)=0) \asa P(f,l,M)<_{0^{*}} P(g,l,M).
\ee
\end{princ}
%
\begin{thm}\label{dargooo}
In $\RCA_{0}^{\dagger}+\eqref{S}+\textup{QF-AC}^{1,1}$, $(D_{2}) \asa\textup{(\textsf{SB})}\asa \textup{(\textsf{TB})}\asa \Deltat$.
\end{thm}
\begin{proof}
The proof of Theorem \ref{labbekak} can easily be adapted to yield the equivalence to \textsf{(SB)}.  For the implication $\Deltat\di$ \textsf{(TB)}, if the left-hand side of $D(f,g)$ holds, there is a standard such $h^{1}$ and $P(f,l,M)$ will be the initial 
segment of such a standard function.  Since the right-hand side of $D(f,g)$ holds, we can prove $(\forall k^{0}\leq_{0^{*}}\overline{l}M)(\exists y^{0}\leq M)(g(\overline{k}y)\ne0)$ 
(using $\paai$ and \eqref{S}) in the same way as in the proof of Theorem~\ref{labbekak}.  Thus, $P(g,l,M)$ is $M_{0}M_{0}\dots M_{0}M_{0}$ by definition, and $P(f,l,M)<_{0^{*}}P(g,l,M)$ follows.  
The reverse implication in \textsf{(TB)} follows similarly from $\Deltat$.  The implication \textsf{(TB)} $\di \Deltat$ follows in the same way as for \textsf{(SB)}, i.e.\ as in the proof of Theorem \ref{labbekak}, by noting that $P(f,l,M)$ in the right-hand side of \eqref{bound6} must be a sequence other than $M_{0}\dots M_{0}$, implying that $(\exists k^{0}\leq_{0^{*}}\overline{l}M)(\forall x^{0}\leq M)(f(\overline{k}x)=0)$.      
\end{proof}
Comparing \eqref{karf} and \eqref{bound6}, we note that stratified Nonstandard Analysis allows us to treat \emph{type zero quantifiers as `one-dimensional' bounded searches}, 
and \emph{type one quantifiers as `two-dimensional' bounded searches}.  Furthermore, as discussed in Remark \ref{genpb}, the search can be adapted to allow any parameter.        
Note that the right-hand side of \eqref{bound6} now plays the role of the algorithm $\mathfrak{(A)}$ from Section \ref{main1}.  

\subsection{Conclusion}
We now formulate the conclusion of this paper.  In particular, we exhibit the similarity between the notions `computable' (in the form $\Delta_{1}^{0}$) and `$\Delta_{1}^{1}$', and `Turing jump' and `hyperjump'.

\medskip

First of all, in light of Theorem \ref{labbekak}, the principle \textsf{(PB)} implies that for standard $f^{1}$, $h_{0}$ as in Theorem \ref{hake}, and $M^{0}\sler h_{0}$, we have
\be\label{bound2}
(\exists g^{1})(\forall x^{0})(f(\overline{g}x)=0) \asa (\exists g^{0}\leq_{0^{*}}\overline{h_{0}}M)(\forall x^{0}\leq M)(f(\overline{g}x)=0).
\ee
Hence, if we know that $(\exists g^{1})(\forall x^{0})(f(\overline{g}x)=0)$, then \eqref{bound2} tells us that a `two-dimensional' bounded search (involving the bounds $\overline{h_{0}}M$ and $M\sler h_{0}$) 
will yield a sequence $\sigma^{0}$ of length $M$ such that $(\forall^{\st} x^{0})(f(\overline{\sigma}x)=0)$.  By Corollary \ref{FRS}, we can find such a sequence \emph{with a standard part} using the algorithm $\mathfrak{(A)}$.  
By \eqref{STP} and $\paai$, the output of $(\mathfrak{A})$ then has a unique standard part $g^{1}$, which is such that $(\forall x^{0})(f(\overline{g}x)=0)$.  The search performed by the algorithm $\mathfrak{(A)}$ is similar to that associated to \eqref{tranke}, i.e.\ the Turing (resp.\ hyper-) jump corresponds to a one- (resp two-) dimensional bounded search.  
In both cases, an instance of the Transfer principle derives from the Turing- and hyperjump, and this principle is needed to certify that the associated search provides the correct output.          

\medskip

Secondly, by Theorem \ref{dargooo}, a similar result is available for $\Delta_{1}^{1}$-formulas, analogous to the case of $\Delta_{1}^{0}$-formulas.  Indeed, to verify if a $\Delta_{1}^{0}$-formula as in \eqref{valid} (with `st' removed) holds for some $n_{0}$, 
one checks, one by one, the following sequence:
\be\label{furg}
f(n_{0},0)=0, g(n_{0}, 0)\ne 0,f(n_{0},1)=0, g(n_{0}, 1)\ne 0,\dots, 
\ee
which by definition yields a terminating search, and gives rise to $p(\cdot, M)$ in \eqref{karf}.  The latter is similar to $P(\cdot, l, M)$ from \eqref{bound6}, and one can perform a search similar to \eqref{furg} for a $\Delta_{1}^{1}$-formula as in $D(f,g)$ by checking 
$(\forall x\leq M)f(\overline{\rho}x)=0$ and $(\exists y\leq M)g(\overline{\rho}x)=0$ for $\rho^{0}$ equal to $00\dots0 0$, $00\dots 01$, $00\dots 02$, et cetera, where all sequences have length $M$ (and are below $\overline{h_{0}}M$ from Theorem \ref{hake}).  Thus, verifying if a $\Delta_{1}^{0}$ (resp.\ $\Delta_{1}^{1}$) formula holds, corresponds to a double one- (resp two-) dimensional bounded search involving $p(\cdot, M)$ (resp.\ $P(\cdot, h_{0},M)$).

\medskip

Thirdly, we should stress that the right-hand sides of \eqref{bound}, \eqref{bound5}, \eqref{bound6}, and \eqref{bound2} do satisfy our conditions \eqref{conda} and \eqref{condb}.  
Indeed, as pointed out above, the function $h_{0}$ from Theorem \ref{hake} is constructively acceptable, and there is a clear similarity between the Kleene normal form and the bounded formulas.  

\medskip

In conclusion, stratified Nonstandard Analysis allows us to treat \emph{type zero quantifiers as `one dimensional' bounded searches}, 
and \emph{type one quantifiers as `two dimensional' bounded searches}.  
In particular, in light of Nelson's dictum from Remark \ref{ohdennenboom} that every specific object of conventional mathematics is a standard set, it seems that \textsf{(PB)}, \textsf{(SB)}, and \textsf{(TB)} allow us to search through the reals for internal properties 
\emph{involving parameters from conventional mathematics}, which is quite a rich world.  By Remark \ref{genpb}, the search can be adapted to allow any parameter.       

\begin{ack}\rm
This research was supported by the following funding bodies: FWO Flanders, the John Templeton Foundation, the University of Oslo, the Alexander von Humboldt Foundation, and the Japan Society for the Promotion of Science.  
The author expresses his gratitude towards these institutions. 
The author would like to thank Karel Hrbacek, Dag Normann, and Toby Meadows for their valuable advice.  
\end{ack}


\begin{bibdiv}
\begin{biblist}
\bib{aczelrathjen}{book}{
  author={Aczel, Peter},
  author={Rathjen, Michael},
  title={Notes on Constructive Set Theory},
  series={Reports Institut Mittag-Leffler}, volume={40},
  date={2000/2001},
}
\bib{avi3}{article}{
  author={Avigad, Jeremy},
  title={Weak theories of nonstandard arithmetic and analysis},
note={See \cite{simpson1}},
}

\bib{avi2}{article}{
  author={Avigad, Jeremy},
  author={Feferman, Solomon},
  title={G\"odel's functional \(``Dialectica''\) interpretation},
  conference={ title={Handbook of proof theory}, },
  book={ series={Stud. Logic Found.}, volume={137}, publisher={North-Holland}, },
  date={1998},
  pages={337--405},
}

\bib{brie}{article}{
  author={van den Berg, Benno},
  author={Briseid, Eyvind},
  author={Safarik, Pavol},
  title={A functional interpretation for nonstandard arithmetic},
  journal={Ann. Pure Appl. Logic},
  volume={163},
  date={2012},
  number={12},
  pages={1962--1994},
}

\bib{bennosam}{article}{
  author={van den Berg, Benno},
  author={Sanders, Sam},
  title={Transfer equals Comprehension},
  journal={Submitted},
  volume={},
  date={2014},
  number={},
  note={Available on arXiv: \url {http://arxiv.org/abs/1409.6881}},
  pages={},
}

\bib{briebenno}{article}{
  author={van den Berg, Benno},
  author={Briseid, Eyvind},
  title={Weak systems for nonstandard arithmetic},
  journal={In preparation},
}

\bib{bish1}{book}{
  author={Bishop, Errett},
  title={Foundations of constructive analysis},
  publisher={McGraw-Hill Book Co.},
  place={New York},
  date={1967},
  pages={xiii+370},
}

\bib{aveirohrbacek}{article}{
  author={Hrbacek, Karel},
  title={Stratified analysis?},
  conference={ title={The strength of nonstandard analysis}, },
  book={ publisher={Springer}, },
  date={2007},
  pages={47--63},
}

\bib{hrbacek3}{article}{
  author={Hrbacek, Karel},
  title={Relative Set Theory: Internal View},
  journal={J. Log. Anal.},
  volume={1},
  date={2009},
  pages={Paper 8, pp.\ 108},
  issn={1759-9008},
}

\bib{hrbacek4}{article}{
  author={Hrbacek, Karel},
  author={Lessmann, Olivier},
  author={O'Donovan, Richard},
  title={Analysis with ultrasmall numbers},
  journal={Amer. Math. Monthly},
  volume={117},
  date={2010},
  number={9},
  pages={801--816},
}

\bib{hrbacek5}{article}{
  author={Hrbacek, Karel},
  title={Relative Set Theory: Some external issues},
  journal={J. Log. Anal.},
  volume={2},
  date={2010},
  pages={pp.\ 37},
}

\bib{kaye}{book}{
  author={Kaye, Richard},
  title={Models of Peano arithmetic},
  series={Oxford Logic Guides},
  volume={15},
  publisher={The Clarendon Press},
  date={1991},
  pages={x+292},
}

\bib{keisler1}{article}{
  author={Keisler, H. Jerome},
  title={Nonstandard arithmetic and reverse mathematics},
  journal={Bull. Symb.\ Logic},
  volume={12},
  date={2006},
  pages={100--125},
}

\bib{kohlenbach2}{article}{
  author={Kohlenbach, Ulrich},
  title={Higher order reverse mathematics},
note={See \cite{simpson1}},
}

\bib{ML}{book}{
   author={Martin-L{\"o}f, Per},
   title={Intuitionistic type theory},
   series={Studies in Proof Theory. Lecture Notes},
   volume={1},
   publisher={Bibliopolis},
   date={1984},
   pages={iv+91},
}

\bib{wownelly}{article}{
  author={Nelson, Edward},
  title={Internal set theory: a new approach to nonstandard analysis},
  journal={Bull. Amer. Math. Soc.},
  volume={83},
  date={1977},
  number={6},
  pages={1165--1198},
}

\bib{peraire}{article}{
  author={P{\'e}raire, Yves},
  title={Th\'eorie relative des ensembles internes},
  language={French},
  journal={Osaka J. Math.},
  volume={29},
  date={1992},
  number={2},
  pages={267--297},
}

\bib{robinson1}{book}{
  author={Robinson, Abraham},
  title={Non-standard analysis},
  publisher={North-Holland},
  place={Amsterdam},
  date={1966},
  pages={xi+293},
}

\bib{yamayamaharehare}{article}{
  author={Sakamoto, Nobuyuki},
  author={Yamazaki, Takeshi},
  title={Uniform versions of some axioms of second order arithmetic},
  journal={MLQ Math. Log. Q.},
  volume={50},
  date={2004},
  number={6},
  pages={587--593},
}

\bib{firstHORM}{article}{
  author={Sanders, Sam},
  title={Uniform and nonstandard existence in Reverse Mathematics},
  year={2014},
  journal={Submitted, Available from arXiv: \url {http://arxiv.org/abs/1502.03618}},
}

\bib{simpson1}{collection}{
   title={Reverse mathematics 2001},
   series={LNL},
   volume={21},
   editor={Simpson, Stephen G.},
   publisher={ASL},
   date={2005},
   pages={x+401},
}


\bib{simpson2}{book}{
  author={Simpson, Stephen G.},
  title={Subsystems of second order arithmetic},
  series={Perspectives in Logic},
  edition={2},
  publisher={CUP},
  date={2009},
  pages={xvi+444},
}

\bib{zweer}{book}{
  author={Soare, Robert I.},
  title={Recursively enumerable sets and degrees},
  series={Perspectives in Mathematical Logic},
  publisher={Springer},
  date={1987},
  pages={xviii+437},
}

\bib{troelstra1}{book}{
  author={Troelstra, Anne Sjerp},
  title={Metamathematical investigation of intuitionistic arithmetic and analysis},
  note={Lecture Notes in Mathematics, Vol.\ 344},
  publisher={Springer Berlin},
  date={1973},
  pages={xv+485},
}


\end{biblist}
\end{bibdiv}
\bye